\documentclass{amsart}
\usepackage{amscd}
\usepackage{amsmath,empheq}
\usepackage{amsfonts}
\usepackage{amssymb}
\usepackage{mathrsfs}

\usepackage[all]{xy}

\newtheorem{theorem}{Theorem}

\newtheorem{prop}[theorem]{Proposition}
\newtheorem{remark}{Remark}

\newenvironment{proof-sketch}{\noindent{\bf Sketch of Proof}\hspace*{1em}}{\qed\bigskip}

\DeclareMathOperator*{\essinf}{ess\,inf}

\everymath{\displaystyle}

\newcommand{\RR}{\mathbb R}
\newcommand{\NN}{\mathbb N}

\renewcommand{\leq}{\leqslant}

\renewcommand{\geq}{\geqslant}

\begin{document}

\title[Singular Dirichlet problems]{Positive solutions for a class of singular Dirichlet problems}

\author[N.S. Papageorgiou]{Nikolaos S. Papageorgiou}
\address[N.S. Papageorgiou]{National Technical University, Department of Mathematics,
				Zografou Campus, Athens 15780, Greece \& Institute of Mathematics, Physics and Mechanics, 1000 Ljubljana, Slovenia}
\email{\tt npapg@math.ntua.gr}

\author[V.D. R\u{a}dulescu]{Vicen\c{t}iu D. R\u{a}dulescu}
\address[V.D. R\u{a}dulescu]{Faculty of Applied Mathematics, AGH University of Science and Technology, 30-059 Krak\'ow, Poland \& Institute of Mathematics, Physics and Mechanics, 1000 Ljubljana, Slovenia}
\email{\tt vicentiu.radulescu@imfm.si}

\author[D.D. Repov\v{s}]{Du\v{s}an D. Repov\v{s}}
\address[D.D. Repov\v{s}]{Faculty of Education and Faculty of Mathematics and Physics, University of Ljubljana \& Institute of Mathematics, Physics and Mechanics, 1000 Ljubljana, Slovenia}
\email{\tt dusan.repovs@guest.arnes.si}


\begin{abstract}
We consider a Dirichlet elliptic problem driven by the Laplacian with singular and superlinear nonlinearities. The singular term appears on the left-hand side while the superlinear perturbation is parametric with parameter $\lambda>0$ and it need not satisfy the AR-condition. Having as our starting point the work of Diaz-Morel-Oswald (1987), we show that there is a critical parameter value $\lambda_*$ such that for all $\lambda>\lambda_*$ the problem has two positive solutions, while for $\lambda<\lambda_*$ there are no positive solutions. What happens in the critical case $\lambda = \lambda_*$ is an interesting open problem.
\end{abstract}

\maketitle

\section{Introduction}

Let $\Omega \subseteq \RR^{N}(\NN \geq 2)$ be a bounded domain with a $C^2$-boundary $\partial\Omega$. In this paper we study the following parametric singular Dirichlet problem
\begin{equation}
-\Delta u(z)+ u(z)^{-\gamma} = \lambda f(z,u(z)) \ \mbox{in} \ \Omega, \  u|_{\partial\Omega} = 0,\ u>0,\ 0<\gamma<1.
\tag{$P_\lambda$}\label{eqp}
\end{equation}

In this problem, $\lambda$ is a positive parameter and $f:\Omega \times \RR \rightarrow \RR$ is a Carath\'eodory function (that is, for all $x\in\RR$, the mapping $z\mapsto f(z,x)$ is measurable and for almost all $z \in \Omega $, the mapping $x\mapsto f(z,x)$ is continuous). We assume that for almost all $z\in \Omega$, $ f(z,\cdot)$ exhibits superlinear growth near $+\infty$, but it need not satisfy the usual in such cases Ambrosetti-Robinowitz condition (the AR-condition for short).

The distinguishing feature of our work is that the singular term $u^{-\gamma}$ appears on the left-hand side of the equation. This is in contrast with almost all previous works on singular elliptic equations driven by the Laplacian, where the forcing term (the right-hand side of the equation) is $u\mapsto u^{-\gamma}+\lambda f(z,u)$, so the singular term $u^{-\gamma}$ appears on the right-hand side of the equation. We mention the works of Coclite \& Palmieri \cite{2}, Sun, Wu \& Long \cite{13}, and Haitao \cite{7}, which also deal with equations that have the competing effects of singular and superlinear terms. A comprehensive bibliography on semilinear singular Dirichlet problems can be found in the book by Ghergu \& R\u{a}dulescu \cite{5}. The present class of singular equations was first considered by Diaz, Morel \& Oswald \cite{3}, for the case when the perturbation $f$ is independent of $u$. They produced a necessary and sufficient condition for the existence of positive solutions in terms of the integral $\int_{\Omega} f\hat{u}_1dz$, with $\hat{u}_1$ being the positive $L^2$-normalized principal eigenfunction of $(-\Delta, H_0^1 (\Omega))$. More recently, Papageorgiou \& R\u{a}dulescu \cite{11} considered problem \eqref{eqp} with $f(z,\cdot)$ being sublinear.

Our aim is to study the precise dependence of the set of positive solutions of problem \eqref{eqp} with respect to the parameter $\lambda>0$. In this direction, we show that there exists a critical parameter value $\lambda_* >0$ such that
\begin{itemize}
	\item for all $\lambda > \lambda_*$, problem \eqref{eqp} has at least two positive smooth solutions;
	\item for all $\lambda\in (0,\lambda_*)$, problem \eqref{eqp} has no positive solutions.
\end{itemize}

It is an open problem what happens in the critical case $\lambda = \lambda_*$. We describe the difficulties one encounters when treating the critical case $\lambda = \lambda_*$ and why we think that $\lambda_*>0$ is not admissible.

\section{Preliminaries and hypotheses}

Let $X$ be a Banach space and $\varphi \in C^1 (X,\RR)$. We say that $\varphi(\cdot)$ satisfies the ``C-condition", if the following property holds
$$
\begin{array}{ll}
``\mbox{Every sequence}\ \{u_n\}_{n\geq1}\subseteq X\ \mbox{such that}\\
\{\varphi(u_n)\}_{n\geq1}\subseteq\RR\ \mbox{is bounded and}
\ (1+||u_n||_X)\varphi'(u_n)\rightarrow0\ \mbox{in}\ X^*\ \mbox{as}\ n\rightarrow\infty,\\
\mbox{admits a strongly convergent subsequence}".
\end{array}
$$

This is a compactness-type condition on the functional $\varphi$  and it leads to the minimax theory of the critical values of $\varphi$ (see, for example, Papageorgiou, R\u{a}dulescu \& Repov\v{s} \cite{12}).

The main spaces used in the analysis of problem \eqref{eqp} are the Sobolev space $H_0^1(\Omega)$ and the Banach space $C_0^1 (\overline{\Omega})=\{ u\in C^1(\overline{\Omega}): u|_{\partial\Omega} = 0\}$. By $||\cdot||$ we denote the norm of $H_0^1(\Omega)$. On account of the Poincar\'e inequality we have
$$ ||u|| = ||Du||_2\ \mbox{for all}\ u\in H_0^1(\Omega). $$

The Banach space $C_0^1(\overline{\Omega})$ is an ordered Banach space with positive cone $$C_+ = \{ u\in C_0^1(\overline{\Omega}): u(z) \geq 0\ \mbox{for all}\ z \in \overline{\Omega}\}.$$ This cone has a nonempty interior given by
$$
{\rm int}\, C_+ = \left\{u\in C_+ : u(z) >0 \ \mbox{for all} \ z\in \Omega , \frac{\partial u}{\partial n} |_{\partial \Omega} <0 \right\},
$$
with $n(\cdot)$ being the outward unit normal on $\partial \Omega$.

We will also use two other ordered Banach spaces, namely  $C^1(\overline{\Omega})$ and $$C_0(\overline{\Omega}) =\{ u \in C(\overline{\Omega}): u |_{\partial \Omega} =0\}.$$ The order cones are $$\hat{C}_+ = \{ u\in C^1 (\overline{\Omega}):u(z)\geq 0\ \mbox{for all}\ z \in\overline{\Omega}\}$$ and $$K_+ = \{ u\in C_0 (\overline{\Omega}):u(z) \geq 0\ \mbox{for all}\ z\in \overline{\Omega}\},$$ respectively. Both have nonempty interiors given by
$$
\begin{array}{ll}
	D_+ = \{ u\in \hat{C}_+ : u(z) > 0 \ \mbox{for all} \ z\in \overline{\Omega}\},\\
	\mathring{K}_+ = \{ u\in K_+ : c_u \hat{d} \leq u \ \mbox{for some}\ c_u>0\},
\end{array}
$$
with $\hat{d}(z)=d(z,\partial\Omega)$ for all $z\in\overline{\Omega}$.

Concerning ordered Banach spaces, the following result is helpful (see Papageorgiou, R\u{a}dulescu \& Repov\v{s} \cite[Proposition 4.1.22, p. 226]{12}).
\begin{prop}\label{prop1}
	If $X$ is an ordered Banach space with order (positive) cone $K$, ${\rm int}\,K \neq \emptyset$, and $e\in {\rm int}\,K$, then for every $u \in X$, we can find $\lambda_u >0$ such that $\lambda_u e -u\in K$.
\end{prop}
	
	Next, we introduce the main notation which we will use in the sequel. Given $\varphi \in C^1(H_0^1(\Omega))$, we denote by $K_\varphi$ the critical set of $\varphi$, that is,
	$$
	K_\varphi = \{ u\in H_0^1 (\Omega): \varphi'(u) = 0 \}.
	$$

For $x\in \RR$, we set $x^\pm = \max\{\pm x,0\}$. Given $u\in H_0^1(\Omega)$, we set $u^\pm (z) = u(z)^\pm$ for all $z\in \Omega$. We know that
$$
	u^\pm \in H_0^1(\Omega),\ u=u^+ - u^- ,\ |u|=u^++u^-.
$$

	Given $u , y  \in H_0^1(\Omega)$ with $u\leq y$, we define
$$
\begin{array}{ll}
	&[u,y] = \{ h\in H_0^1(\Omega): u(z)\leq h(z)\leq y(z)\ \mbox{for almost all}\ z\in \Omega \},\\
	&[u) = \{ h\in H_0^1 (\Omega): u(z)\leq h(z)\ \mbox{for almost all}\ z\in \Omega \}.
\end{array}
$$

Also, by
$$
{\rm int}_{C_0^1(\overline{\Omega})} [u,y]
$$
we denote the interior of $[u,y]\cap C_0^1 (\overline{\Omega})$  in the $C_0^1(\overline{\Omega})$-norm topology.

By $\hat{\lambda}_1 >0$ we denote the principal eigenvalue of $(-\Delta, H_0^1(\Omega))$ and by $\hat{u}_1$ the corresponding positive $L^2$-normalized (that is, $||\hat{u}_1||_2 = 1$) eigenfunction. Standard regularity theory and the Hopf maximum principle imply that $\hat{u}_1\in  {\rm int}\,C_+$.

Finally, by $2^*$ we denote the critical Sobolev exponent, $2^* = \left\{
		\begin{array}{ll}
			\frac{2N}{N-2}&\mbox{if}\ N\geq 3 \\
			+\infty&\mbox{if}\ N=2
		\end{array}
	\right.$.	
	
Now we will introduce our hypotheses on the perturbation $f(z,x)$.

\smallskip
$H(f):f:\Omega \times \RR \rightarrow \RR$ is a Carath\'eodory function such that $f(z,0)=0$ for almost all $z\in \Omega$ and
\begin{itemize}
	\item [(i)] $f(z,x) \leq a(z) (1+x^{r-1})$ for almost all $z\in\Omega$ and all $x\geq 0$, with $a\in L^{\infty}(\Omega)$, $2<r<2^*$;
	\item [(ii)] if $F(z,x) = \int_{0}^x f(z,s)ds$, then $\lim_{x\to +\infty} \frac{F(z,x)}{x^2} = +\infty$ uniformly for almost all $z\in \Omega$;
	\item [(iii)] there exists $\tau \in (r-2,2^*)$ such that
	 	$$
	 	0 < \beta_0 \leq \liminf_{x\to \infty} \frac{f(z,x)x- 2F(z,x)}{x^\tau}\ \mbox{uniformly for almost all}\ z\in \Omega;
	 	$$
	 \item [(iv)] for every $\rho>0$ and every $\lambda>0$, there exists $\hat{\xi}_\rho^\lambda>0$ such that for almost all $z\in\Omega$, the function
	 	$$
	 	x\mapsto \lambda f(z,x) + \hat{\xi}_\rho^{\lambda}x
	 	$$
	 	is nondecreasing on $[0,\rho]$ and for every $s>0$ we have
	 	$$
	 	\inf \{f(z,x): x\geq s\} = m_s >0\ \mbox{for almost all}\ z\in \Omega;
		$$
	 \item [(v)] there exist $q>2$, $\delta_0>0,\hat{c}>0$ such that
	 	$$
	 	\begin{array}{ll}
	 		\hat{c} x^{q-1} \leq f(z,x) \ \mbox{for almost all}\ z\in \Omega \ \mbox{and all}\ 0 \leq x \leq \delta_0,\\
	 		\lim_{x\to 0^+} \frac{F(z,x)}{x^2} = 0 \ \mbox{uniformly for almost all}\ z\in \Omega.
	 	\end{array}
	 	$$
\end{itemize}

\begin{remark}
	Since we are looking for positive solutions and all of the above hypotheses concern the positive semiaxis $\RR_+ = [0,+\infty)$,  we may assume without any loss of generality that
\begin{equation}\label{eq1}
	f(z,x)=0 \ \mbox{for almost all}\ z\in \Omega\ \mbox{and all}\ x \leq 0.
\end{equation}
\end{remark}

Hypotheses H(f)(ii),(iii) imply that
$$ \lim_{x\to +\infty} \frac{f(z,x)}{x} = +\infty \ \mbox{uniformly for almost all}\ z\in \Omega .$$
So, the perturbation $f(z,\cdot)$ is superlinear. However, we do not express this superlinearity of $f(z,\cdot)$ by using the traditional (for superlinear problems) AR-condition. We recall that the AR-condition (the unilateral version due to \eqref{eq1}) says that there exist $\vartheta>2$ and $M>0$ such that
\begin{subequations}
	\begin{align}
	&0<\vartheta F(z,x) \leq f(z,x) x \ \mbox{for almost all}\ \ z\in \Omega\ \mbox{and all}\ x \geq M \label{eq2a}\\
	&0 < \essinf_{\Omega} F(\cdot,M). \label{eq2b}
	\end{align}
\end{subequations}

Integrating \eqref{eq2a} and using \eqref{eq2b}, we obtain the following weaker condition
\begin{eqnarray*}
	&&c_1 x^\vartheta \leq F(z,x)\ \mbox{for almost all}\ z\in \Omega,\ \mbox{all}\ x \geq M,\ \mbox{and some}\ c_1 >0,\\
	&\Rightarrow &c_1 x^{\vartheta -1} \leq f(z,x)\ \mbox{for almost all}\ z\in \Omega\ \mbox{and all}\ x \geq M \ \mbox{(see \eqref{eq2a})}.
\end{eqnarray*}

So, the AR-condition dictates at least $(\vartheta - 1)$-polynomial growth for $f(z,\cdot)$. Here, instead of the AR-condition, we employ hypothesis $H(f)(iii)$ which is less restrictive and incorporates in our framework superlinear nonlinearities with ``slower" growth near $+\infty$. Consider the following function (for the sake of simplicity we drop the z-dependence)
$$
f(x) = \left\{
		\begin{array}{ll}
			c x^{q-1}&\mbox{if}\ 0 \leq x \leq 1\\
			x \ln x + c x^{\vartheta-1}&\mbox{if}\ 1<x
		\end{array}
	\right.
$$
with $c>0,q>2>\vartheta>1$ (see \eqref{eq1}). Then $f(\cdot)$ satisfies hypotheses $H(f)$ but it fails to satisfy the AR-condition.

Finally, we mention that for $u\in H_0^1(\Omega)$ we have
\begin{equation}
	c_u \hat{d} \leq u\ \mbox{for some}\ c_u>0\ \mbox{if and only if}\ \hat{c}_u\hat{u}_1 \leq u\ \mbox{for some}\ \hat{c}_u >0.
	\label{eq3}
\end{equation}

For $\delta>0$ let $\Omega_\delta = \{ z\in \Omega: d(z,\partial \Omega) < \delta \}$ and let $\tilde{C}^1(\Omega_\delta) = \{ u\in C^1(\overline{\Omega}_\delta): u|_{\partial \Omega} = 0\}$ with order cone $$\tilde{C}^1(\overline{\Omega}_\delta)_+ = \{ u \in \tilde{C}^1 (\overline{\Omega}_\delta):u(z) \geq 0\ \mbox{for all}\ z\in \overline{\Omega}_\delta\},$$ which has nonempty interior given by $${\rm int}\, \tilde{C}^1(\overline{\Omega}_\delta)_+ = \left\{ u\in \tilde{C}^1 (\overline{\Omega}_\delta)_+:u(z)>0\ \mbox{for all}\ z\in \Omega_\delta\ \mbox{and}\ \frac{\partial u}{\partial n}|_{\partial\Omega}<0 \right\}.$$ 
According to Lemma 14.16 of Gilbarg \& Trudinger \cite[p. 355]{6}, for $\delta >0$ small enough we have $\hat{d}\in {\rm int}\, \tilde{C}^1(\overline{\Omega}_\delta)_+.$ Also, we have $\hat{d}\in D_+(\overline{\Omega}\backslash \Omega_\delta)$, with the latter being the interior of the order cone of $C^1(\overline{\Omega}\backslash \Omega_\delta)$. So, using Proposition \ref{prop1} we can find $0<\hat{c}_1<\hat{c}_2$ such that $\hat{c_1}\hat{d}\leq \hat{u}_1\leq \hat{c}_2\hat{d}$ (recall that $\hat{u}_1 \in {\rm int}\,C_+$). This implies \eqref{eq3}.

\section{Positive solutions}

Let $\eta >0$. We start by considering the following auxiliary purely singular Dirichlet problem
\begin{equation}
	-\Delta u(z) + u(z)^{-\gamma} = \eta \hat{u}_1(z) \ \mbox{in}\ \Omega ,\ u|_{\partial \Omega} = 0
	\tag*{$(Au)_\eta$}.\label{eqAu}
\end{equation}

By Theorem 1 of Diaz, Morel \& Oswald \cite{3} we know that for $\eta >0$ big problem \ref{eqAu} has a solution $v_\eta\in H_0^1(\Omega)\cap C_0(\overline{\Omega})$ and $v^{-\gamma}_\eta \in L^1(\Omega), c_\eta \hat{u}_1 \leq v_\eta$ for some $c_\eta >0.$

Also, we consider the following Dirichlet problem
\begin{equation}
	-\Delta u(z) = \lambda f(z,u(z)) \ \mbox{in}\ \Omega,\ u|_{\partial\Omega},\ u>0,\ \lambda>0.
	\tag{$Q_\lambda$}\label{eqq}
\end{equation}

\begin{prop}\label{prop2}
	If hypotheses $H(f)$ hold and $\lambda>0$, then problem \eqref{eqq} has a solution $\hat{u}_\lambda\in {\rm int}\, C_+$.
\end{prop}
\begin{proof}

	Let $\Psi_\lambda : H_0^1 (\Omega) \rightarrow\RR$ be the $C^1$- functional defined by
	$$ \Psi_\lambda(u) = \frac{1}{2} ||Du||_2^2 - \int_{\Omega}\lambda F(z,u) dz\ \mbox{for all}\ u\in H_0^1(\Omega). $$
	
	Hypotheses $H(f)(ii),(iii)$ imply that
	\begin{equation}
		\Psi_\lambda(\cdot) \ \mbox{satisfies the C-condition}
		\label{eq4}
	\end{equation}
(see Papageorgiou \& R\u{a}dulescu \cite[Proposition 9]{10}).

Combining hypotheses $H(f)(i),(v)$, given $\epsilon >0$, we can find $c_\epsilon >0$ such that
$$ F(z,x) \leq \frac{\epsilon}{2	} x^2 +c_{\epsilon}x^r\ \mbox{for almost all}\ z\in \Omega, \ \mbox{all}\ x\geq0. $$
	
Then we have
\begin{equation}
	\begin{array}{ll}
		\Psi_\lambda (u) &\geq \frac{1}{2} ||Du||_2^2-\frac{\lambda\epsilon}{2} ||u||_2^2 -\lambda \hat{c}_\epsilon ||u||^r\ \mbox{for some}\ \hat{c}_\epsilon>0\\
		&\geq c_2 ||u||^2 -\lambda\hat{c}_\epsilon ||u||^r\ \mbox{for some}\ c_2=c_2 (\lambda)>0 \ \mbox{(choose}\ \epsilon>0 \ \mbox{small enough}),\\
		&\Rightarrow u=0 \ \mbox{is a local minimizer of}\ \Psi_\lambda(\cdot)\ \mbox{(recall that}\ r>2).
		\label{eq5}
	\end{array}
\end{equation}

We can easily see that if $u\in K_{\Psi_\lambda}$, then $u\geq0$. Hence we assume that $K_{\Psi_\lambda}$ is finite. On account of \eqref{eq5} and Theorem 5.7.6 of Papageorgiou, R\u{a}dulescu \& Repov\v{s} \cite[p. 367]{12}, we can find $\rho \in (0,1)$ so small that
\begin{equation}
	0 = \Psi_\lambda(0) < \inf\{\Psi_\lambda(u):||u||=\rho\} = m_\lambda.
	\label{eq6}
\end{equation}

Hypothesis $H(f)(ii)$ implies that
\begin{equation}
	\Psi_\lambda(t\hat{u}_1) \rightarrow -\infty \ \mbox{as}\ t\rightarrow +\infty.
	\label{eq7}
\end{equation}

Then \eqref{eq4}, \eqref{eq6}, \eqref{eq7} permit the use of the mountain pass theorem. So, we can find $\hat{u}_\lambda \in H_0^1(\Omega)$ such that
$$
\begin{array}{ll}
	\hat{u}_\lambda \in K_{\Psi_\lambda} \ \mbox{and}\ m_\lambda \leq \Psi_\lambda (\hat{u}_\lambda),\\
	\Rightarrow \hat{u}_\lambda	 \geq 0,\ \hat{u}_\lambda \neq 0\ \mbox{(see \eqref{eq6})}.
\end{array}
$$

We have
$$
\begin{aligned}
	&\int_{\Omega}(D\hat{u}_\lambda,Dh)_{\RR^N}dz = \lambda \int_{\Omega}f(z,u_\lambda) hdz \ \mbox{for all}\ h\in H_0^1(\Omega),\\
	\Rightarrow &-\Delta u_\lambda(z) = \lambda f(z,u_\lambda(z)) \geq 0\ \mbox{for almost all}\ z\in\Omega.
\end{aligned}
$$

Then the semilinear regularity theory (see Gilbarg \& Trudinger \cite{6}) and the Hopf maximum principle (see Gasinski \& Papageorgiou \cite{4}), imply that $\hat{u}_\lambda \in {\rm int}\,C_+$.
\end{proof}

Hypotheses $H(f)$ imply that we can find $c_2>0$ such that
\begin{equation}
	f(z,x)\geq c_2 min\{ x,x^{q-1}\} \ \mbox{for almost all}\ z\in\Omega \ \mbox{and all}\ x\geq 0.
	\label{eq8}	
\end{equation}

We have $\hat{u}_\lambda \in {\rm int}\,C_+$ and $\hat{u}_\lambda^{q-1}\in{\rm int}\,K_+$. So, we can find $c_3>0$ such that
\begin{equation}
	\begin{array}{ll}
		&\eta \hat{u}_1 \leq c_3 \hat{u}_\lambda \ \mbox{and}\ \eta\hat{u}_1 \leq c_3 \hat{u}_\lambda^{q-1},\\
		\Rightarrow &\eta\hat{u}_1 \leq c_3 \min\{\hat{u}_\lambda,\hat{u}_\lambda^{q-1} \}.
	\end{array}
	\label{eq9}
\end{equation}

From \eqref{eq8} and \eqref{eq9} we see that we can find $\lambda_0 \geq 0$ big such that for $\lambda \geq \lambda_0$
\begin{equation}
	\begin{array}{ll}
		\lambda f(z,\hat{u}_\lambda(z)) \geq \lambda_0c_2\, \min\{ \hat{u}_\lambda (z), \hat{u}_\lambda(z)^{q-1}\}\\
		\geq \eta \hat{u}_1(z) \ \mbox{for almost all}\ z\in\Omega.
	\end{array}
	\label{eq10}	
\end{equation}

Recall that $c_\eta \hat{u}_1 \leq v_\eta$ for some $c_\eta>0$. Hence by \eqref{eq3}, $\hat{c}_\eta \hat{d} \leq v_\eta$ for some $\hat{c}_\eta>0$. Therefore
$$ v_\eta^{-\gamma} \leq \frac{1}{\hat{c}_\eta^{\gamma}}\hat{d}^{-\gamma}. $$

For every $h\in H_0^1(\Omega)$, we have
$$
\begin{array}{ll}
	\int_{\Omega}\frac{|h|}{v_\eta^{\gamma}} dz \leq \frac{1}{\hat{c}_\eta^{\gamma}} \int_{\Omega}\frac{|h|}{\hat{d}^{\gamma}} dz\\
	= \frac{1}{\hat{c}_\eta^{\gamma}} \int_{\Omega} \frac{|h|}{\hat{d}} \hat{d}^{1-\gamma} dz \leq c_4 \int_{\Omega}\frac{|h|}{\hat{d}} dz\ \mbox{for some}\ c_4>0.
\end{array}
$$

Invoking Hardy's inequality (see Brezis \cite[p. 313]{1}), we infer that $\frac{|h|}{\hat{d}}\in L^2(\Omega)$. Therefore
$$
\begin{array}{ll}
	&c_4\int_{\Omega} \frac{|h|}{\hat{d}}dz \leq c_5 \left( \int_{\Omega}\frac{|h|^2}{\hat{d}^2}dz \right)^{\frac{1}{2}} <\infty \ \mbox{for some}\ c_5>0,\\
	\Rightarrow &| \int_{\Omega}v_\eta^{-\gamma}hdz |<\infty\ \mbox{for all}\ h\in H_0^1(\Omega).
\end{array}
$$

Therefore we have
\begin{equation}
	-\Delta v_\eta (z) + v_\eta (z)^{-\gamma} = \eta \hat{u}_1(z)\ \mbox{for almost all}\ z\in\Omega.
	\label{eq11}	
\end{equation}

If $\lambda \geq \lambda_0$, from \eqref{eq10} and \eqref{eq11} we have
\begin{equation}
	\begin{array}{rr}
	-\Delta \hat{u}_\lambda(z) = \lambda f(z,\hat{u}_\lambda(z)) \geq \eta \hat{u}_1(z) = -\Delta v_\eta(z) + v_\eta(z)^{-\gamma} \geq -\Delta v_\eta(z)\\
	\mbox{for almost all}\ z\in\Omega.
	\end{array}
	\label{eq12}	
\end{equation}

Since $v_\eta |_{\partial\Omega} = \hat{u}_\lambda|_{\partial\Omega}=0$, from \eqref{eq12} and the weak comparison principle (see Tolksdorf \cite[Lemma 3.1]{14}), we have
\begin{equation}
	v_\eta \leq \hat{u}_\lambda\quad (\lambda \geq \lambda_0).
	\label{eq13}	
\end{equation}

Now we introduce the following two sets
$$
\begin{array}{ll}
	&\mathcal{L} = \{ \lambda>0	:\ \mbox{problem \eqref{eqp} has a positive solution} \},\\
	&S_\lambda = \ \mbox{the set of positive solutions of \eqref{eqp}}.
\end{array}
$$

Here by a solution of \eqref{eqp}, following \cite{11}, we understand a function $u\in H_0^1(\Omega)$ such that
\begin{itemize}
	\item [(a)]	$u\in L^{\infty}(\Omega), u(z) >0$ for almost all $z\in\Omega$, and $u^{-\gamma}\in L^1(\Omega)$;
	\item [(b)] there exists $c_u>0$ such that $c_u\hat{d}\leq u;$ \label{itb}
	\item [(c)] $\int_{\Omega} (Du,Dh)_{\RR^N}dz + \int_{\Omega}u^{-\gamma}hdz = \lambda \int_{\Omega} f(z,u)hdz$ for all $h\in H_0^1(\Omega).$
\end{itemize}

From \eqref{eq3} we know that (b) is equivalent to saying that $\hat{c}_u\hat{u}_1 \leq u$ for some $\hat{c}_u>0$. Also the previous discussion reveals that (c) makes sense. Regularity theory will provide additional structure for the solutions of \eqref{eqp}.

\begin{prop}\label{prop3}
	If hypotheses $H(f)$ hold, then $\mathcal{L} \neq \emptyset$ and $S(\lambda) \subseteq{\rm int}\, C_+$.
\end{prop}

\begin{proof}
	Let $\lambda \geq \lambda_0$. Using \eqref{eq13} we can introduce the Carath\'eodory function $g_\lambda(z,x)$ defined by
\begin{equation}
 g_\lambda(z,x) = \left\{
	 \begin{array}{ll}
			\lambda f(z,v_\eta(z)) - v_\eta(z)^{-\gamma}&\mbox{if}\ x < v_\eta(z)\\
			\lambda f(z,x) - x^{-\gamma} &\mbox{if}\ v_\eta(z)\leq x \leq \hat{u}_\lambda(z) \\
			\lambda f(z,\hat{u}_\lambda(z)) - \hat{u}_\lambda(z)^{-\gamma}&\mbox{if}\ \hat{u}_\lambda(z) < x.
		\end{array}
	\right.
	\label{eq14}
\end{equation}

We set $G_\lambda(z,x)= \int_{0}^x g_\lambda(z,s)ds$ and consider the functional $\varphi_\lambda: H_0^1(\Omega)\rightarrow\RR$ defined by
$$ \varphi_\lambda(u) = \frac{1}{2}||Du||_2^2 - \int_{\Omega} G_\lambda(z,u)dz\ \mbox{for all}\ u\in H_0^1(\Omega). $$

From Papageorgiou \& R\u{a}dulescu \cite{11} (see Claim 1 in the proof of Proposition \ref{prop6}), we have that $\varphi_\lambda \in C^1(H_0^1(\Omega))$. It is clear from \eqref{eq14} that $\varphi_\lambda(\cdot)$ is coercive. Also, it is sequentially weakly lower semicontinuous. Therefore we can find $u_\lambda\in H_0^1(\Omega)$ such that
\begin{equation}
	\begin{array}{ll}
		&\varphi_\lambda(u_\lambda) = \inf \left\{\varphi_\lambda(u):u\in H_0^1(\Omega)\right\},\\
	 	\Rightarrow & \varphi_\lambda^{'}(u_\lambda)=0,\\
	 	\Rightarrow & \int_{\Omega}(Du_\lambda,Dh)_{\RR^{N}}dz = \int_{\Omega}g_\lambda(z,u_\lambda)hdz\ \mbox{for all}\ h\in H_0^1(\Omega).
	\end{array}
	\label{eq15}
\end{equation}

In \eqref{eq15} first we choose $h=(u_\lambda -\hat{u}_\lambda)^+\in H_0^1(\Omega)$. We have
$$
\begin{array}{ll}
	\int_{\Omega} (Du_\lambda,D(u_\lambda - \hat{u}_\lambda)^+)_{\RR^{N}}dz &= \int_{\Omega} [\lambda f(z,\hat{u}_\lambda)-\hat{u}_\lambda^{-\gamma}](u_\lambda-\hat{u}_\lambda)^+ dz \ \mbox{(see \eqref{eq14}})\\
	&\leq \int_{\Omega} \lambda f(z,\hat{u}_\lambda)(u_\lambda - \hat{u}_\lambda)^+ dz\\
	&= \int_{\Omega} (D\hat{u}_\lambda,D(u_\lambda-\hat{u}_\lambda)^+)_{\RR^{N}}dz \ \mbox{(see Proposition \ref{prop2}}),\\
	\Rightarrow ||D(u_\lambda-\hat{u}_\lambda)^+||_2^2 \leq 0,\\
	\Rightarrow u_\lambda \leq \hat{u}_\lambda.
\end{array}
$$

Next, in \eqref{eq15} we choose $h=(v_\eta-u_\lambda)^+\in H_0^1(\Omega).$ Then we have
\begin{equation}
	\int_{\Omega} (D u_\lambda,D(v_\eta- u_\lambda)^+)_{\RR^{N}} dz = \int_{\Omega}[\lambda f(z,v_\eta)- v_\eta^{-\gamma}] (v_\eta-u_\lambda)^+ dz.
	\label{eq16}	
\end{equation}

As we proved \eqref{eq10}, using \eqref{eq8}, \eqref{eq9}, we see that by taking $\lambda \geq \lambda_0$ even bigger if necessary, we can have
\begin{equation}
	\lambda f(z,v_\eta(z)) \geq \eta\hat{u}_1(z)\ \mbox{for almost all}\ z\in\Omega.
	\label{eq17}	
\end{equation}

Hence from \eqref{eq16} and \eqref{eq17} we have
$$
\begin{array}{ll}
	\int_{\Omega}(Du_\lambda,D(v_\eta-u_\lambda)^+)_{\RR^{N}} dz &\geq \int_{\Omega} [\eta\hat{u}_1 -v_\eta^{-\gamma}](v_\eta-u_\lambda)^+ dz\\
	&= 	\int_{\Omega}(Dv_\eta,D(v_\eta-u_\lambda)^+)_{\RR^{N}} dz\\
	\Rightarrow ||D(v_\eta-u_\lambda)^+||_2^2 \leq 0,\\
	\Rightarrow v_\eta \leq u_\lambda.
\end{array}
$$

So, we have proved that
\begin{equation}
	u_\lambda\in[v_\eta,\hat{u}_\lambda].
	\label{eq18}	
\end{equation}

It follows from \eqref{eq14}, \eqref{eq15} and \eqref{eq18} that
$$ \int_{\Omega}(Du_\lambda,Dh)_{\RR^{N}}dz + \int_{\Omega}u_\lambda^{-\gamma}hdz = \int_{\Omega}f(z,u_\lambda)hdz\ \mbox{for all}\ h\in H_0^1(\Omega). $$

Recall that
$$
\begin{array}{ll}
	&c_\eta \hat{d} \leq v_\eta \leq u_\lambda \\
\mbox{and}\ &u_\lambda^{-\gamma}\leq v_\eta^{-\gamma}\in L^1(\Omega)\ \mbox{(see \eqref{eq18})}.
\end{array}
$$

Therefore $u_\lambda$ is a solution of \eqref{eqp}. We have proved that for $\lambda \geq \lambda_0$ big enough, we have $\lambda\in \mathcal{L}$ and so $\mathcal{L}\neq \emptyset$.

Now let $u\in S_\lambda$. Then by definition we have
\begin{equation}
	\begin{array}{ll}
		&\hat{c}_u\hat{u}_1\leq u\ \mbox{for some}\ \hat{c}_u>0,\\
		\Rightarrow &u\in{\rm int}\,K_+.
	\end{array}
	\label{eq19}	
\end{equation}

Let $s>N$. Since $\hat{u}_1^{1/s}\in K_+$, we can find $c_6>0$ such that
$$
\begin{array}{ll}
	&{\hat{u}_1}^{1/s} \leq c_6 u\ \mbox{(see Proposition \ref{prop1})},\\
	\Rightarrow &u^{-\gamma} \leq c_7{\hat{u}_1}^{-\frac{\gamma}{s}}\ \mbox{for some}\ c_7>0.	
\end{array}
$$

However, by  Lemma in Lazer \& McKenna \cite{8}, we have that $\hat{u}_1^{-\frac{\gamma}{s}}\in L^s(\Omega)$ (recall that $0<\gamma<1$). So, it follows that $u^{-\gamma}\in L^s(\Omega).$ Then Theorem 9.15 of Gilbarg \& Trudinger \cite[p. 241]{6} implies that $u\in W^{2,s}(\Omega)$. Since $s>N$, from the Sobolev embedding theorem, we have $u\in C^{1,\alpha}(\overline{\Omega})$ with $\alpha = 1 -\frac{N}{s}.$ We conclude that $u\in{\rm int}\,C_+$ (see \eqref{eq19}) and so $S_\lambda \subseteq{\rm int}\,C_+$.
\end{proof}

Next, we prove a structural property for the set $\mathcal{L}$ and a kind of monotonicity property for the set $S_\lambda$ with respect to $\lambda\in\mathcal{L}$.

\begin{prop}\label{prop4}
	If hypotheses $H(f)$ hold, $\lambda\in\mathcal{L},\ \mu>\lambda$, and $u_\lambda\in S_\lambda \subseteq{\rm int}\,C_+,$ then $\mu\in\mathcal{L}$ and we can find $u_\mu\in S_\mu\subseteq{\rm int}\,C_+$.	
\end{prop}

\begin{proof}
	Let $\rho=||u_\lambda||_\infty.$ Hypotheses $H(f)$ imply that we can find $c_\rho>0$ such that
	\begin{equation}
		0\leq f(z,x)\leq c_\rho x\ \mbox{for almost all}\ z\in\Omega\ \mbox{and all}\ 0\leq x\leq\rho.
		\label{eq20}
	\end{equation}
	
	Also from \eqref{eq8} we know that
	\begin{equation}
		f(z,x)\geq c_2 \min\{x,x^{q-1}\}\ \mbox{for almost all}\ z\in\Omega\ \mbox{and all}\ x\geq0.
		\label{eq21}	
	\end{equation}

	Recall that for $\vartheta \geq \lambda_0$ we have $\hat{u}_\theta \geq v_\eta$ (see \eqref{eq13}) and $v_\eta\in{\rm int}\,K_+$. So, for $\vartheta \geq \lambda_0$ big enough we have
	\begin{equation}
		\vartheta c_2\min \{ \hat{u}_\theta,\hat{u}_\theta^{q-1}\} \geq \lambda c_\rho u_\lambda.
		\label{eq22}	
	\end{equation}

	It follows that
	$$
	\begin{array}{ll}
		-\Delta\hat{u}_\theta = \vartheta f(z,\hat{u}_\theta) &\geq \vartheta c_2\min \{\hat{u}_\theta,\hat{u}_\theta^{q-1}\} \ \mbox{(see \eqref{eq21})}\\
		&\geq \lambda c_\rho u_\lambda \ \mbox{(see \eqref{eq22})}\\
		&\geq \lambda f(z,u_\lambda) \ \mbox{(see \eqref{eq20})}\\
		&= -\Delta u_\lambda +u_\lambda^{-\gamma}\ \mbox{(since}\ u_\lambda \in S_\lambda)\\
		&\geq -\Delta u_\lambda \ \mbox{for almost all}\ z\in\Omega,\\
	\end{array}
	$$
	$$ \Rightarrow \hat{u}_\theta \geq u_\lambda\ \mbox{(by the weak comparison principle, see Tolksdorf \cite{14})}. $$
	
	Therefore we can introduce the Carath\'eodory function $k_\mu(z,x)$ defined by
\begin{equation}
		k_\mu(z,x) = \left\{
			\begin{array}{ll}
				\mu f(z,u_\lambda(z)) - u_\lambda(z)^{-\gamma}&\mbox{if}\ x<u_\lambda(z)\\
				\mu f(z,x)-x^{-\gamma}&\mbox{if}\ u_\lambda(z)\leq x\leq \hat{u}_\theta(z)\\
				\mu f(z,\hat{u}_\theta(z)) -\hat{u}_\theta(z)^{-\gamma}&\mbox{if}\ \hat{u}_\theta(z)<x.
			\end{array}
		\right.
	\label{eq23}
\end{equation}
	
	We set $K_\mu(z,x)=\int_{0}^{x} k_\mu(z,s)ds$ and consider the functional $\sigma_\mu :H_0^1(\Omega)\rightarrow\RR$ defined by
	$$ \sigma_\mu(u) = \frac{1}{2}||Du||_2^2 - \int_\Omega k_\mu(z,u)dz \ \mbox{for all}\ u\in H_0^1(\Omega). $$
	
	Again we have $\sigma_\mu \in C^1 (H_0^1(\Omega))$ (see Papageorgiou \& R\u{a}dulescu \cite{11}). From (\ref{eq23}) it is clear that $\sigma_\mu(\cdot)$ is coercive. Also, by the Sobolev embedding theorem we see that $\sigma_\mu(\cdot)$ is sequentially weakly lower semicontinuous. So, by the Weierstrass-Tonelli theorem, we can find $u_\mu\in H_0^1(\Omega)$ such that
	$$
	\begin{array}{ll}
		&\sigma_\mu(u_\mu) = \inf \left\{ \sigma_\mu(u): u\in H_0^1(\Omega)\right\},\\
		\Rightarrow &\sigma_\mu'(u_\mu)=0,\\
		\Rightarrow &\int_{\Omega}(Du_\lambda,Dh)_{\RR^{N}} dz = \int_{\Omega}k_\mu(z,u_\lambda)hdz \ \mbox{for all}\ h\in H_0^1(\Omega).
	\end{array}
	$$
	
	Choosing first $h=(u_\mu-\hat{u}_{\theta})^+ \in H_0^1(\Omega)$ and then $h=(u_\lambda-u_\mu)^+ \in H_0^1(\Omega)$ as in the proof of Proposition \ref{prop3}, we can show that
	\begin{equation}
		\begin{array}{ll}
		&u_\mu \in [u_\lambda,\hat{u}_\theta],\\
		\Rightarrow & u_\mu\in S_\mu \subseteq{\rm int}\,C_+ \ \mbox{(see \eqref{eq23})}.
		\end{array}
		\label{eq24}
	\end{equation}
	
	Let $\rho=||\hat{u}_\theta||_\infty$ and let $\hat{\xi}_0 = \max \{\hat{\xi}_\rho^\lambda,\hat{\xi}_\rho^\mu$\} (see hypothesis $H(f)(iv))$. We have
	$$
	\begin{array}{ll}
		-\Delta u_\lambda +\hat{\xi}_0 u_\lambda &= \lambda f(z,u_\lambda)+ \hat{\xi}_0 u_\lambda -u_\lambda^{-\gamma}\\
		&\leq \mu f(z,u_\mu) + \hat{\xi}_0 u_\mu -u_\mu^{-\gamma}\\
		&\mbox{(see hypothesis }\ H(f)(iv)\ \mbox{and \eqref{eq24})}\\
		&=-\Delta u_\mu + \hat{\xi}_0 u_\mu \ \mbox{(since}\ u_\mu\in S_\mu),\\
		\end{array}
	$$
	$$
	\begin{array}{ll}
		\Rightarrow \Delta(u_\mu-u_\lambda) \leq \hat{\xi}_0 (u_\mu- u_\lambda),\\
		\Rightarrow u_\mu - u_\lambda \in{\rm int}\,C_+ \ \mbox{(by Hopf's maximum principle)}.
	\end{array}
 	$$
	The proof is now complete.
\end{proof}

This proposition implies that $\mathcal{L}$ is a half-line. More precisely, let $\lambda_*=inf \mathcal{L}$. We have
\begin{equation}
	(\lambda_*,+\infty)\subseteq \mathcal{L} \subseteq [\lambda_*,+\infty).
	\label{eq25}	
\end{equation}

\begin{prop}\label{prop5}
	If hypotheses $H(f)$ hold, then $\lambda_*>0.$	
\end{prop}

\begin{proof}
	Arguing by contradiction, suppose that $\lambda_*=0$. Let $\left\{\lambda_n\right\}_{n\geq 1} \subseteq \mathcal{L}$ such that $\lambda_n \downarrow 0$ and let $u_n \in S_{\lambda_n} \subseteq{\rm int}\,C_+$ for all $n\in\NN$. We know that
	\begin{equation}
		\begin{array}{ll}
			0\leq u_n\leq \hat{u}_\theta \ \mbox{for}\ \vartheta\geq\lambda_0\ \mbox{big enough, for all}\ n\in\NN  \\
			\mbox{(see the proof of Proposition \ref{prop4})},
		\end{array}
		\label{eq26}	
	\end{equation}
	\begin{equation}
		\begin{array}{ll}
			-\Delta u_n +u_n^{-\gamma} = \lambda_n f(z,u_n) \ \mbox{for almost all}\ z\in\Omega\ \mbox{and all}\ n\in\NN.
		\end{array}
		\label{eq27}	
	\end{equation}
	
	Let $\eta>0$. With $\rho =||\hat{u}_\theta||_\infty$ (see \eqref{eq26}), we have
	\begin{equation}
		\begin{array}{ll}
			-\Delta u_n + u_n^{-\gamma} &= \lambda_n f(z,u_n)\\
			&\leq \lambda_n c_\rho u_n \ \mbox{(see \eqref{eq20})}\\
			&\leq \lambda_n c_\rho \hat{u}_\theta \ \mbox{(see \eqref{eq26})}\\
			&\leq \eta\hat{u}_1\ \mbox{for all}\ n\geq n_0 \ \mbox{(recall that }\ \hat{u}_1\in{\rm int}\,C_+).
		\end{array}
		\label{eq28}	
	\end{equation}

By \eqref{eq28} and Theorem 1(i) of Diaz, Morel \& Oswald \cite{3} it follows that problem \ref{eqAu}  has a positive solution. Since $\eta>0$ is arbitrary, we contradict Theorem 1(ii) of Diaz, Morel \& Oswald \cite{3}. This proves that $\lambda_*>0$.
\end{proof}

\begin{prop}\label{prop6}
	If hypotheses $H(f)$ hold and $\lambda_*<\lambda$, then problem \eqref{eqp} has at least two positive solutions $u_0,\,\hat{u}\in{\rm int}\,C_+,\, u_0 \neq \hat{u}$.	
\end{prop}

\begin{proof}
	Let $\lambda_*<\sigma<\lambda<\mu.$ On account of Proposition \ref{prop4}, we can find $u_\sigma\in S_\sigma \subseteq{\rm int}\,C_+, u_0 \in S_\lambda \subseteq{\rm int}\, C_+\ \mbox{and}\  u_\mu\in S_\lambda \subseteq{\rm int}\,C_+$\ such that
	\begin{equation}
		\begin{array}{ll}
			&u_0 - u_\sigma \in{\rm int}\,C_+\ \mbox{and}\ u_\mu-u_0 \in{\rm int}\,C_+,\\
			\Rightarrow &u_0\in{\rm int}_{C_0^1(\overline{\Omega})} [u_\sigma,u_\mu].
		\end{array}
		\label{eq29}	
	\end{equation}

We introduce the Carath\'eodory functions $e_\lambda(z,x)$ and $\hat{e}_\lambda(z,x)$ defined by
\begin{equation}
	e_\lambda(z,x) = \left\{
		\begin{array}{ll}
			\lambda f(z,u_\sigma(z)) - u_\sigma(z)^{-\gamma}&\mbox{if}\ x\leq u_\sigma(z)\\
			\lambda f(z,x) -x^{-\gamma}&\mbox{if}\ u_\sigma(z)<x
		\end{array}
	\right.
	\label{eq30}
\end{equation}
\begin{equation}
	\mbox{and}\ \hat{e}_\lambda(z,x) = \left\{
		\begin{array}{ll}
			e_\lambda(z,x)&\mbox{if}\ x\leq u_\mu(z)\\
			e_\lambda(z,u_\mu(z))&\mbox{if}\ u_\mu(z)<x.
		\end{array}
	\right.
	\label{eq31}
\end{equation}

We set $E_\lambda(z,x)=\int_{0}^{x} e_\lambda(z,s)ds$ and $\hat{E}_\lambda(z,x) = \int_{0}^{x} \hat{e}_\lambda(z,s)ds$ and consider the $C^1$-functionals $\beta_\lambda,\hat{\beta}_\lambda : H_0^1(\Omega)\rightarrow \RR$ defined by
$$
\begin{array}{ll}
	\beta_\lambda(u) = \frac{1}{2}||Du||_2^2 - \int_{\Omega} E_\lambda(z,u)dz,\\
	\hat{\beta}_\lambda(u)= \frac{1}{2}||Du||_2^2 - \int_{\Omega}\hat{E}_\lambda(z,u)dz\ \mbox{for all}\ u\in H_0^1(\Omega).
\end{array}
$$

Using \eqref{eq30}, \eqref{eq31}, as before (see the proof of Proposition \ref{prop3}), we can check that
\begin{equation}
	K_{\beta_\lambda} \subseteq [u_\sigma) \cap{\rm int}\,C_+\ \mbox{and}\ K_{\hat{\beta}_\lambda} \subseteq [u_\sigma,u_\mu] \cap{\rm int}\,C_+.
	\label{eq32}	
\end{equation}

Using \eqref{eq32}, \eqref{eq30} and \eqref{eq29}, we see that we may assume that
\begin{equation}
	K_{\beta_\lambda}\ \mbox{is finite and}\ K_{\beta_\lambda} \cap [u_\sigma,u_\mu] = \{u_0\}.
	\label{eq33}	
\end{equation}

Otherwise, we already have additional positive solutions and so we are done.

Evidently $\hat{\beta}_\lambda(\cdot)$ is coercive (see \eqref{eq30}). Also, it is sequentially weakly lower semicontinuous. Thus we can find $\hat{u}_0 \in H_0^1 (\Omega)$ such that
\begin{equation}
	\begin{array}{ll}
		&\hat{\beta}_\lambda (\hat{u}_0) = \inf \left\{ \hat{\beta}_\lambda(u):u\in H_0^1(\Omega)\right\},\\
		\Rightarrow &\hat{u}_0 \in K_{\hat{\beta}_\lambda} \subseteq [u_\sigma,u_\mu] \cap{\rm int}\,C_+\ \mbox{(see \eqref{eq32})}.
	\end{array}
	\label{eq34}	
\end{equation}

From \eqref{eq30} and \eqref{eq31} we see that
$$
\begin{array}{ll}
	&\beta_\lambda^{'}|_{[u_\sigma,u_\mu]} = \hat{\beta}_\lambda^{'}|_{[u_\sigma,u_\mu]},\\
	\Rightarrow &\hat{u}_0 \in K_{\beta_\lambda}\cap[u_\sigma,u_\mu]\ \mbox{(see \eqref{eq34})}\\
	\Rightarrow &\hat{u}_0 = u_0\ \mbox{(see \eqref{eq33})},\\
	\Rightarrow & u_0\ \mbox{is a local}\ C_0^1(\overline{\Omega})\mbox{-minimizer of}\ \beta_\lambda(\cdot),\\
	\Rightarrow & u_0\ \mbox{is a local}\ H_0^1(\Omega)\mbox{-minimizer of}\ \beta_\lambda(\cdot)\ \mbox{(see \cite{9})}.
\end{array}
$$

Then from \eqref{eq33} and Theorem 5.7.6 of Papageorgiou, R\u{a}dulescu \& Repov\v{s} \cite[p. 367]{12}, we know that we can find $\rho\in(0,1)$ so small that
\begin{equation}
	\beta_\lambda(u_0) < \inf \left\{\beta_\lambda(u): ||u-u_0||=\rho\right\}=m_\lambda.
	\label{eq35}	
\end{equation}

Hypothesis $H(f)(ii)$ implies that
\begin{equation}
	\beta_\lambda(t\hat{u}_1) \rightarrow -\infty \ \mbox{as}\ t\rightarrow +\infty.
	\label{eq36}	
\end{equation}

Finally, recall that hypothesis $H(f)(iii)$ implies that
\begin{equation}
	\beta_\lambda(\cdot) \ \mbox{satisfies the C-condition}	
	\label{eq37}
\end{equation}
(see Papageorgiou \& R\u{a}dulescu \cite{10}).

Then \eqref{eq35}, \eqref{eq36}, \eqref{eq37} permit the use of the mountain pass theorem. So, we can find $\hat{u}\in H_0^1(\Omega)$ such that
$$
\begin{array}{ll}
	&\hat{u}\in K_{\beta_\lambda}\ \mbox{and}\ m_\lambda\leq\beta_\lambda(\hat{u}),\\	
	\Rightarrow &\hat{u}\in S_\lambda \subseteq{\rm int}\,C_+, \hat{u} \neq u_0 \ \mbox{(see \eqref{eq32}, \eqref{eq31} and \eqref{eq35})}.
\end{array}
$$
The proof is now complete.
\end{proof}

Summarizing, we can state the following theorem for the set of positive solutions of problem \eqref{eqp}.

\begin{theorem}\label{th7}
	If hypotheses $H(f)$ hold, then there exists $\lambda_*>0$ such that
	\begin{itemize}
		\item [(a)] for all $\lambda > \lambda_*$ problem \eqref{eqp} has at least two positive solutions
			$$ u_0,\hat{u}\in{\rm int}\,C_+,u_0 \neq \hat{u}; $$
		\item [(b)] for all $\lambda\in(0,\lambda_*)$ problem \eqref{eqp} has no positive solutions.
	\end{itemize}
\end{theorem}

\begin{remark}
From the above Theorem is missing what happens at the critical case $\lambda=\lambda_*$. We were unable to resolve this case.
\end{remark}

If $\lambda_n \downarrow \lambda_*$, then we can show that there exist $u_n\in S_{\lambda_n} \subseteq{\rm int}\,C_+$ $(n\in\NN)$ such that 	
$$ u_n \xrightarrow{\text{w}} u_*\ \mbox{in}\ H_0^1(\Omega), u_* \neq 0. $$

As before (see the proof of Proposition \ref{prop3}), we have
$$ u_n^{-\gamma}\in L^s(\Omega)\ (s>N)\ \mbox{and}\ u_n^{-\gamma}\rightarrow u_*^{-\gamma}\ \mbox{for almost all}\ z\in\Omega. $$

However, we can not show that $\{u_n^{\gamma}\}_{n\geq 1}\subseteq L^s(\Omega)$ is bounded  and therefore have that
$$
\begin{array}{ll}
 \int_{\Omega}u_n^{-\gamma}hdz\rightarrow \int_{\Omega} u_*^{-\gamma}hdz\ \mbox{for all}\ h\in H_0^1(\Omega)\\ \\
 \mbox{(Vitali's theorem, see Gasinski \& Papageorgiou \cite[p. 901]{4})}.
\end{array}
$$

In addition, we can not show that there exists $c_*>0$ such that
$$ u_*\geq c_*\hat{d}. $$

It seems that $\lambda_*>0$ is not admissible (that is, $\lambda_* \notin \mathcal{L}$, hence $\mathcal{L}=(\lambda_*,+\infty)$, see \eqref{eq25}), but this needs a proof.

Another open problem is the possibility of extending this work to equations driven by the $p$-Laplacian. This extension requires a corresponding generalization of the work of Diaz, Morel \& Oswald \cite{3} to the case of the $p$-Laplacian. However, the tools of \cite{3} are particular for the Laplacian. So, it is not clear how this generalization can be achieved. Hence new techniques are needed.

\medskip
{\bf Acknowledgments.} This research was supported by the Slovenian Research Agency grants
P1-0292, J1-8131, J1-7025, N1-0064, and N1-0083, and the Romanian Ministry of Research and Innovation, CNCS--UEFISCDI, project number PN-III-P4-ID-PCE-2016-0130,
within PNCDI III.

\end{document}